\documentclass{amsart}
\usepackage{amssymb,latexsym}
\usepackage{enumerate}
\usepackage{color}
\usepackage{dsfont}

\usepackage[centertags]{amsmath}
\usepackage{amsfonts}
\usepackage{amssymb}
\usepackage{amsthm}
\usepackage{newlfont}
\usepackage[all]{xy}
\usepackage{fancyhdr}

\usepackage{pgf,tikz,pgfplots}
\usepackage{mathrsfs}
\usetikzlibrary{arrows}

\newtheorem{theorem}{Theorem}[section]
\newtheorem{lemma}[theorem]{Lemma}

\theoremstyle{definition}

\theoremstyle{remark}

\numberwithin{equation}{section}

\newcommand{\zR}{\mathbb R}
\newcommand{\zE}{\mathbb E}
\newcommand{\ve}{\varepsilon}
\begin{document}

\title[Orthant Probabilities and maxima on a vertex of a simplex]{Orthant Probabilities and the attainment of \\ maxima on a vertex of a simplex}


\author{Dami\'an Pinasco}
\address{Universidad Torcuato Di Tella Av. Figueroa Alcorta 7350 (C1428BCW) Buenos
Aires, Argentina and CONICET}
\email{dpinasco@utdt.edu}
%
\author{Ezequiel Smucler}
\address{Universidad Torcuato Di Tella Av. Figueroa Alcorta 7350 (C1428BCW) Buenos
Aires, Argentina and CONICET}
\email{esmucler@utdt.edu}
%
\author{Ignacio Zalduendo}
\address{Universidad Torcuato Di Tella Av. Figueroa Alcorta 7350 (C1428BCW) Buenos
Aires, Argentina and CONICET}
\email{izalduendo@utdt.edu}
\thanks{Supported partially by PIP 112 201301 00422 CO (CONICET - Argentina).}

\keywords{Homogeneous Polynomial, Orthant Probabilities, Simplex}

\subjclass[2010]{26D05, 46G25, 62H05}

\date{}

\dedicatory{}

\commby{}

\begin{abstract}
We calculate bounds for orthant probabilities for the equicorrelated multivariate normal distribution and use these bounds to
show the following: for degree $k>4$, the probability that a $k$-homogeneous polynomial in $n$ variables attains a relative maximum on a
vertex of the $n$-dimensional simplex tends to one as the dimension $n$ grows. The bounds we obtain for the orthant probabilities are tight up to $\log(n)$ factors.
\end{abstract}

\maketitle
\section*{Introduction}\label{sec1}

A number of papers have appeared addressing the question of the probability that an $n$-variable polynomial $P: \zR^n \longrightarrow \zR$ attain relative maxima on
specific points (\cite{CZ}, \cite{PGV}, \cite{PZ1}, \cite{PZ2}). It has been proved for instance, that for degree $k>2$, the probability that a $k$-homogeneous polynomial
in $n$ variables attain a relative maxima on some vertex of the unit sphere of $\ell_1^n$ tends to one as $n$ increases \cite{PZ1}, and that this is false for $k=2$ \cite{PGV}.

To consider such problems a probability measure is defined on the space of $k$-homogeneous $n$-variable polynomials $\mathcal{P}^k(\zR^n)$ (we recall the necessary constructions in section 1).
Functions $X:\mathcal{P}^k(\zR^n) \longrightarrow \zR$ are then random variables and the questions regarding relative maxima may be set in terms of inequalities involving some of these random variables.

In this paper we consider the case of relative maxima attained on a vertex of the zero-centered simplex. This translates into a classical problem in Statistics, that of computing
{\it orthant probabilities}: let $X_1, \ldots , X_n$ be jointly normally distributed random variables with mean $\mu=0$, standard deviation $\sigma=1$, and $\zE(X_i X_j)=\rho_{ij}$;
then what is the probability that they be simultaneously positive?
Sheppard \cite{Sheppard} proved that $P(X_1 >0, X_2>0 )=1/4 + (2\pi)^{-1}\arcsin(\rho_{12})$. David \cite{David} develops a recursive formula that allows one to compute $P(X_1 >0, \dots,X_n>0 )$ for $n$ odd as a function of orthant probabilities of lower dimension. Applying this formula to $n=3$ yields $P(X_1 >0, X_2>0, X_3>0 )=1/8+1/(4\pi) \left(\arcsin(\rho_{12})+\arcsin(\rho_{13})+\arcsin(\rho_{23}) \right)$. No closed form formulas for $n\geq 4$ are known. Series expansions for the general case were derived in \cite{Ruben}. Steck \cite{Steck} studies the special case in which $\zE(X_i X_j)=\rho$, which is the setting of interest in the present work, and lists a number of interesting properties of the function $f(n,\rho)=P(X_1 >0, \dots,X_n>0 )$, including the identity $f(n,\rho)= \mathbb{E}\left[ \Phi^{n}(Z\sqrt{\rho/(1-\rho)}) \right]$, where $Z$ is a standard normal random variable and $\Phi$ is its cumulative distribution function. In \cite{PZ2} there is a formula for $f(n,\rho)$ in terms of the gaussian measure of a simplex. A more thorough review of results regarding orthant probabilities can be found in \cite{Owen}. In this paper, we will show that for $\rho>0$ and large $n$,  $f(n,\rho)$ behaves essentially like $n^{1-1/\rho}$, except for multiplicative $\log(n)$ factors.

A central issue pertaining to the attainment of maxima at several different vertices is the independence or non-independence of the random variables involved.
This accounts for the essential difference in the case of $k>2$ or $k=2$ mentioned above for the unit sphere of $\ell_1^n$. It is also the main difficulty
in the study of the analogous problem on the unit sphere of $\ell_\infty^n$. As we will show in section 2, the particular geometry of the simplex gives rise to
a situation where the random variables involved are not independent, but their dependence is weak. We use the Devroye-Mehrabian-Reddad total-variation bound \cite{DMR} to approximate
asymptotically the non-independent situation by independent ones.

This paper is organized as follows. In Section 1 we recall \cite{PZ2} the construction of the probability measure in $\mathcal{P}^k(\zR^n)$, and present some important random variables and their
correlations. In Section 2 we present the necessary geometry of the simplex, study the attainment of maxima on a given vertex and give bounds for orthant probabilities. Finally, in
Section 3 we consider the problem of attaining maxima at any vertex and prove that for $k>4$ the probability of attaining a relative maxima at some vertex of the simplex tends to one
as the dimension grows.

\section{Probability in spaces of polynomials}\label{PP}

Given a $k$-homogeneous polynomial $P:{\zR}^n \longrightarrow {\zR}$ there exists a unique symmetric $k$-linear function $\phi$ that satisfies $P(x)=\phi(x,\ldots ,x)$. We treat the space of $k$-homogeneous polynomials over ${\zR}^n$ as the dual of the symmetric tensor product
$\bigotimes_{k,s}{\zR}^n$.
The Hilbert space structure on the full tensor product $\bigotimes_k{\zR}^n$ is given by the inner product
$$
\langle v^1 \otimes \cdots \otimes v^k , w^1 \otimes \cdots \otimes w^k \rangle = \langle v^1 , w^1 \rangle \cdots \langle v^k , w^k \rangle .
$$
See \cite{Dwyer} for details.

Let the symmetrization operator $S:\bigotimes_k{\zR}^n \longrightarrow \bigotimes_k{\zR}^n$ be given by setting its values on a basis as
$$
S(e_{j_1} \otimes \cdots \otimes e_{j_k})=\frac{1}{k!}\sum_\sigma e_{j_\sigma(1)} \otimes \cdots \otimes e_{j_\sigma(k)},
$$
where $\sigma$ runs through all permutations of $\{1,\ldots ,k\}$. The image of $S$ ---the symmetric tensor product $\bigotimes_{k,s}{\zR}^n$---
is a predual of the space of polynomials over ${\zR}^n$, $\mathcal{P}^k({\zR}^n)$. We consider on $\bigotimes_{k,s}{\zR}^n$ the
Hilbert space structure induced by the ambient space $\bigotimes_k {\zR}^n$, and on $\mathcal{P}^k({\zR}^n)$ the dual Hilbert space structure.
The resulting norm on $\mathcal{P}^k({\zR}^n)$ is the Bombieri norm \cite{BBEM}
$$
\Vert P \Vert = \left(\sum_{|\alpha|=k} a_\alpha^2 \frac{\alpha !}{k!}\right)^{1/2},
$$
if $P(x)=\sum_{|\alpha|=k} a_\alpha x^\alpha$ is the monomial-sum expression of $P$.

All linear forms $\varphi : \mathcal{P}^k({\zR}^n) \longrightarrow {\zR}$ can be identified with an element of $\bigotimes_{k,s}{\zR}^n.$
For instance, evaluation at $x$: $e_x(P)=P(x)$ is given by $x \otimes \cdots \otimes x$. In the present work, we will mainly deal with linear forms
$$
\frac{\partial}{\partial v}(a):\mathcal{P}^k({\zR}^n) \longrightarrow {\zR}\quad \text{ given by }P\mapsto \frac{\partial P}{\partial v}(a),
$$
which can be identified with an element of $\bigotimes_{k,s}{\zR}^n$ as follows
$$
\frac{\partial}{\partial v}(a)=v \otimes a \otimes \cdots \otimes a + a \otimes v \otimes a \otimes \cdots \otimes a + \cdots +
a \otimes \cdots \otimes a \otimes v.
$$
The inner products between such linear forms are given by
\begin{align*}
  \left\langle \frac{\partial}{\partial v}(a),\frac{\partial}{\partial w}(b) \right\rangle
  &= \langle v \otimes a \otimes \cdots \otimes a + \cdots + a \otimes \cdots \otimes a \otimes v , \\
  &\quad \,\,\,\, w \otimes b \otimes \cdots \otimes b + \cdots + b \otimes \cdots \otimes b \otimes w \rangle \\
  &= k\langle v,w \rangle\langle a,b \rangle^{k-1}+(k^2-k)\langle v,b \rangle\langle a,w \rangle\langle a,b \rangle^{k-2}.
\end{align*}
If $\langle a,b \rangle=0$, the display above is zero whenever $k>2$. However, when $k=2$ one has $2\langle a,w \rangle \langle v,b \rangle$ which can be
non-zero. In fact, if $k=2$ and $n=2$,
$$
\frac{\partial}{\partial (1,0)}(0,1) \quad \text{ and } \quad \frac{\partial}{\partial (0,1)}(1,0)
$$
are the same linear form.

We consider on $\mathcal{P}^k({\zR}^n)$ the standard Gaussian measure $W$ corresponding to its Hilbert space structure, i.e.,
the measure
$$
W(B)=\frac{1}{(2 \pi)^{d/2}} \int_B e^{-\frac{\Vert P \Vert^2}{2}} \, dP,
\quad \text{ for any Borel set }B\subset \mathcal{P}^k({\zR}^n),
$$
where $d=\binom{n+k-1}{k}$, is the dimension of $\mathcal{P}^k({\zR}^n)$.
$W$ is rotation-invariant. Moreover, if $T: {\zR}^n \rightarrow {\zR}^n$ is an orthogonal transformation, then
$$
\widetilde{T}: \mathcal{P}^k({\zR}^n) \rightarrow \mathcal{P}^k({\zR}^n) \quad \text{ such that }\widetilde{T}(P)=P\circ T
$$
is a measure-preserving map.

Finally, note that if $\varphi : \mathcal{P}^k({\zR}^n) \longrightarrow {\zR}$ is a linear form then $\varphi$ is a normal random variable with zero mean and standard deviation given by $\Vert \varphi \Vert$.

\section{The simplex and orthant probabilities}\label{OP}

We construct the zero-centered $n$-dimensional simplex in the following way. Take
$$
\Delta^+ = \left\{x\in \zR^{n+1} : x_i\geq 0 \text{ for }i=1,\ldots ,n \text{ and }\sum_{i=1}^{n+1} x_i = 1 \right\}.
$$
Note that $\Delta^+$ is the convex hull of the canonical basis $\{e_1,\ldots ,e_{n+1}\}$ of $\zR^{n+1}$. The center of $\Delta^+$ is $z=(\frac{1}{n+1},\ldots ,\frac{1}{n+1})$.
We now translate $\Delta^+$ by subtracting $z$: $\Delta = \Delta^+ - z$. $\Delta$ and $\Delta^+$ are isometric copies of one another, and
$\Delta$ is contained in the subspace
$$
S=\left\{x\in \zR^{n+1} : \sum_{i=1}^{n+1} x_i = 0 \right\}.
$$
We identify $S$ isometrically with $\zR^n$ and call $\Delta$ the zero-centered $n$-dimensional simplex. The identification of $S$ with $\zR^n$ determines how $\Delta$ is immersed in $\zR^n$,
but as long as its center is at $0$, its position will be irrelevant to us, for one can be taken to another by an isometry.

We will need to calculate the angles that two edges form at a vertex, and also the angles between vertices. For this we refer to $\Delta^+$.
Let $a, a_1, \ldots ,a_n$ be the vertices of $\Delta$, and set $v_i=\frac{a - a_i}{\Vert a - a_i \Vert}$ for $i=1,\ldots , n$.

\noindent
$\bullet$ Angles between edges at a vertex:
\begin{align*}
\arccos \frac{\langle v_i,v_j\rangle}{\Vert v_i \Vert \Vert v_j \Vert} =& \arccos \frac{\langle e_k - e_i,e_k - e_j\rangle}{\Vert e_k - e_i \Vert \Vert e_k - e_j \Vert} \\
                                                          =& \arccos \frac{\Vert e_k \Vert^2 - \langle e_i, e_k\rangle - \langle e_k,e_j\rangle + \langle e_i,e_j\rangle}{\Vert e_k - e_i \Vert^2} \\
                                                          =& \arccos \frac{1}{2} \\
                                                          =& 60^\circ.
\end{align*}

\noindent
$\bullet$ Angles (at zero) between vertices:
\begin{align*}
\arccos \frac{\langle a_i,a_j\rangle}{\Vert a_i \Vert \Vert a_j \Vert} =& \arccos \frac{\langle e_i - z,e_j - z\rangle}{\Vert e_i - z \Vert \Vert e_j - z \Vert} \\
                                                          =& \arccos \frac{\langle e_i,e_j\rangle - \langle e_i, z\rangle - \langle z,e_j\rangle + \Vert z \Vert^2}{\Vert e_i - z \Vert^2} \\
                                                          =& \arccos \frac{-\frac{2}{n+1}+\frac{1}{n+1}}{\frac{n}{n+1}} \\
                                                          =& \arccos \frac{-1}{n}.
\end{align*}

\noindent
We will also need to calculate $\langle v_i, a \rangle$ and $\Vert a \Vert$:
\begin{align*}
\langle v_i,a\rangle = \left\langle \frac{e_k - e_i}{\Vert e_k - e_i \Vert}, e_k - z \right\rangle =&\frac{1}{\sqrt{2}} \left(1 - \langle e_k,z\rangle - \langle e_i,e_k\rangle + \langle e_i, z\rangle \right) =\frac{1}{\sqrt{2}},
\end{align*}

\begin{align*}
\Vert a \Vert^2 = \langle e_k - z,e_k - z\rangle = \langle e_k,e_k\rangle - \langle e_k, z\rangle - \langle z,e_k\rangle + \Vert z \Vert^2
        =& 1 - \frac{2}{n+1} + \frac{1}{n+1} \\
        =& \frac{n}{n+1}.
\end{align*}

Thus, the zero-centered $n$-dimensional simplex has $n+1$ vertices, in each of which incide $n$ edges at $60^\circ$ angles between each other. The angles formed (at zero) by different
vertices are all equal to $\arccos \frac{-1}{n}$, and are therefore closer and closer to perpendicular as the dimension $n$ grows. This will be an important feature for our calculations.

A homogeneous polynomial $P:\zR^n \longrightarrow \zR$ will have (relative to $\Delta$) a maximum at the vertex $a$ if its gradient is ``outward pointing'',

\vskip7mm
\centerline{
\includegraphics[height=8cm]{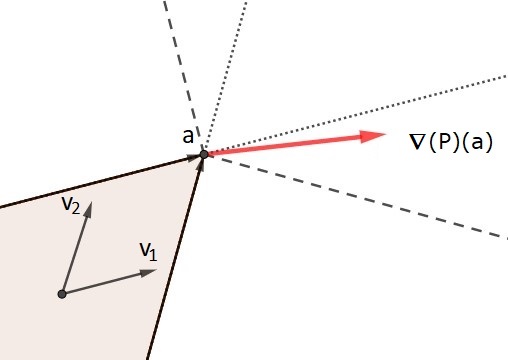}}
\vskip7mm

\noindent
that is, if $\frac{\partial P}{\partial v_i}(a)=\nabla P(a)(v_i) > 0$ for $i=1,\ldots , n$. Thus our problem of relative maximum at $a$ translates into orthant probability:
we want the measure of the subset of $\mathcal{P}^k({\zR}^n)$ on which
$$
X_1(a)=\frac{\frac{\partial}{\partial v_1}(a)}{\Vert \frac{\partial}{\partial v_1}(a) \Vert}, \ldots , X_n(a)=\frac{\frac{\partial}{\partial v_n}(a)}{\Vert \frac{\partial}{\partial v_n}(a) \Vert}
$$
are simultaneously positive. In order to calculate the covariance matrix of the random vector $(X_{1}(a),\dots, X_{n}(a))$ we need the following.
\begin{align*}
\left\Vert \frac{\partial}{\partial v_i}(a) \right\Vert^2 =& \left\langle \frac{\partial}{\partial v_i}(a) , \frac{\partial}{\partial v_i}(a) \right\rangle \\
                                                          =& k\langle v_i,v_i\rangle\langle a,a\rangle^{k-1} + (k^2 - k)\langle v_i,a\rangle^2 \langle a,a\rangle^{k-2} \\
                                                          =& k \Vert a \Vert^{2k-2} + (k^2 - k)\frac{1}{2} \Vert a \Vert^{2k-4} \\
                                                          =& k \Vert a \Vert^{2k-4} \left(\Vert a \Vert^2 + \frac{k-1}{2}\right).
\end{align*}
\begin{align*}
\left\langle \frac{\partial}{\partial v_i}(a) , \frac{\partial}{\partial v_j}(a) \right\rangle =& k\langle v_i,v_j\rangle\langle a,a\rangle^{k-1} + (k^2 - k)\langle v_i,a\rangle\langle a,v_j\rangle \langle a,a\rangle^{k-2} \\
                                                                                               =& \frac{1}{2} k \Vert a \Vert^{2k-2} + (k^2 - k)\frac{1}{2} \Vert a \Vert^{2k-4} \\
                                                                                               =& k \Vert a \Vert^{2k-4} \left(\frac{1}{2} \Vert a \Vert^2 + \frac{k-1}{2}\right).
\end{align*}

In section 3 we will consider maxima at different vertices, $a$ and $b$. For this, we need the following bound (here the $w_j$s define the edges inciding at $b$).

\begin{align*}
\left|\left\langle \frac{\partial}{\partial v_i}(a) , \frac{\partial}{\partial w_j}(b) \right\rangle \right|=& \left|k\langle v_i,w_j\rangle\langle a,b\rangle^{k-1} + (k^2 - k)\langle v_i,b\rangle\langle a,w_j\rangle \langle a,b\rangle^{k-2}\right| \\
                                                 =& \left|k\langle v_i,w_j\rangle \left(\Vert a \Vert^2 \frac{-1}{n}\right)^{k-1} + (k^2 - k)\langle v_i,b\rangle\langle a,w_j\rangle \left(\Vert a \Vert^2 \frac{-1}{n}\right)^{k-2}\right| \\
                                             \leq & k \Vert a \Vert^{2k-2} \left(\frac{1}{n}\right)^{k-1} + (k^2 - k)\Vert a \Vert^2 \Vert a \Vert^{2k-4}\left(\frac{1}{n}\right)^{k-2} \\
                                                 =& k \Vert a \Vert^{2k-4} \left(\Vert a \Vert^2 \frac{1}{n^{k-1}} + (k-1)\Vert a \Vert^2 \frac{1}{n^{k-2}}\right).
\end{align*}
Thus, after normalization, we have $\Vert X_i(a) \Vert = 1$, so $X_i(a)$ are $N(0,1)$ random variables, with correlation
\begin{align}
\nonumber
\langle X_i(a), X_j(a) \rangle =& \frac{k \Vert a \Vert^{2k-4} \left(\frac{1}{2} \Vert a \Vert^2 + \frac{k-1}{2}\right)}{k \Vert a \Vert^{2k-4} \left(\Vert a \Vert^2 + \frac{k-1}{2}\right)} \\\nonumber
                               =& \frac{\Vert a \Vert^2 + (k-1)}{2 \Vert a \Vert^2 + (k-1)} \\\nonumber
                               =& \frac{nk+(k-1)}{n(k+1)+(k-1)} \\
                               \label{eq:rho_n}
                               =& \rho_n.
\end{align}
Note that $\frac{1}{2} < \rho_n < 1$ and $\rho_{n} \to k/(k+1)$. We also have the bound
\begin{align*}
|\langle X_i(a), X_j(b) \rangle| \leq & \frac{\Vert a \Vert^2 \left( \frac{1}{n^{k-1}} + (k-1) \frac{1}{n^{k-2}}\right)}{\Vert a \Vert^2 + \frac{k-1}{2}} \\
                                     =& \frac{1}{n^{k-2}} \frac{1}{2k-1} \left(\frac{1}{n} + (k-1)\right).
\end{align*}

Recall that for $x>0$ the gamma function is given by
$$
\Gamma(x)=\int\limits_{0}^{\infty} t^{x-1} e^{-t} dt.
$$
For $a,b >0$ the beta function is given by
$$
B(a,b)=\int\limits_{0}^{1} t^{a-1} (1-t)^{b-1} dt.
$$
It is well known that
$$
B(a,b)= \frac{\Gamma(a)\Gamma(b)}{\Gamma(a+b)}.
$$
For $Z\sim N(0,1)$ we let $\Phi$ be the cumulative distribution function of $Z$. We will also use the notation $\phi(x)=\Phi^{\prime}(x)$.
We now prove the following bounds for orthant probabilities.
\begin{theorem}\label{cotainf}
Let $\rho \in (0,1)$ and let $(X_1,\ldots ,X_n)$ be a multivariate normal random vector with zero means, variances equal to one and correlations equal to $\rho$. Let $f(n,\rho)=P(X_1 >0, \dots,X_n>0 )$.
\begin{enumerate}
\item There exists a universal constant $c_{1}>0$ such that if $\rho>1/2$ and $n\geq 2$, then
\begin{equation}
\frac{ c_{1}}{\sqrt{2+(1/\rho-1)\log(n)}}  \leq \frac{f(n,\rho)}{n^{1-1/\rho}}.
\label{eq:bounds_rho_high_1}
\end{equation}
and
\begin{equation}
 \frac{f(n,\rho)}{n^{1-1/\rho}} \leq  \frac{1} {2^{2-1/\rho}}\sqrt{\frac{1-\rho}{\rho  }} \left\lbrace 1+B(2, 1/\rho -1)\right\rbrace.
\label{eq:bounds_rho_high_2}
\end{equation}
\item If $\rho<1/2$ and  $n$ satisfies $n\geq (1/\rho -1)\log(n)/\log(2)$, then
\begin{equation}
\frac{1} {2^{2-1/\rho}}\sqrt{\frac{1-\rho}{\rho  }}\left\lbrace \Gamma(1/\rho -1) -1 \right\rbrace \leq \frac{f(n,\rho)}{n^{1-1/\rho}}.
\label{eq:bounds_rho_low}
\end{equation}
Moreover, there exists $n_{0}=n_{0}(\rho)$ such that if $n\geq n_{0}$ then
\begin{equation}
\frac{f(n,\rho)}{n^{1-1/\rho} } \leq   \sqrt{\frac{1-\rho}{\rho}} \left[(1/\rho - 1){\log(n)^{2}}\right]^{1/\rho -2}.
\label{eq:bounds_rho_low_2}
\end{equation}
\end{enumerate}

\end{theorem}

\begin{proof}
Set $s=\rho / (1-\rho)$. By equation (2.5) in \cite{Steck}
\begin{equation}
f(n,\rho) = \mathbb{E}\left[ \Phi^{n}\left(Z \sqrt{s} \right)\right].
\label{eq:steck}
\end{equation}

We begin with the proof of part (1). Note that by Gordon's inequality \cite{Gordon}, for all positive $x$,
$$
1-\Phi(x) < \frac{\phi(x)}{x}
$$
and that by Birnbaum's inequality \cite{Birnbaum}
$$
\frac{2\phi(x)}{\sqrt{4+x^{2}}+x} < 1-\Phi(x).
$$
Let $\alpha_{n}=n^{1/\rho-1}$. By Markov's inequality we have
\begin{align*}
\mathbb{E}\left[ \Phi^{n}\left(Z \sqrt{s} \right)\right]  &\geq \Phi^{n} \left(\sqrt{2\log(\alpha_{n})s}\right) P\left( \Phi^{n}\left(Z \sqrt{s} \right)
\geq \Phi^{n} \left(\sqrt{2\log(\alpha_{n})s}\right)\right)
\\
&= \Phi^{n} \left(\sqrt{2\log(\alpha_{n})s}\right)P\left( Z\geq  \sqrt{2\log(\alpha_{n})}\right).
\end{align*}
Using Birnbaum's lower bound we obtain
\begin{align*}
P\left( Z\geq  \sqrt{2\log(\alpha_{n})}\right) &\geq \frac{2\phi(\sqrt{2\log(\alpha_{n})})}{\sqrt{4+2\log(\alpha_{n})}+\sqrt{2\log(\alpha_{n})}} \\
                                               &=\frac{2}{\sqrt{2\pi}}\frac{1}{\alpha_{n}}\frac{1}{\sqrt{4+2\log(\alpha_{n})}+\sqrt{2\log(\alpha_{n})}}.
\end{align*}
Recall that $\alpha_{n}=n^{1/\rho-1}$, so
$$
P\left( Z\geq  \sqrt{2\log(\alpha_{n})}\right)\geq n^{1-1/\rho} \frac{2}{\sqrt{2\pi}}\frac{1}{\sqrt{4+2(1/\rho-1)\log(n)}+\sqrt{2(1/\rho-1)\log(n)}}.
$$

On the other hand, using Gordon's upper bound, we obtain
\begin{align*}
\Phi^{n} \left(\sqrt{2\log(\alpha_{n})s}\right) \geq \left\lbrace 1 - \frac{\phi(\sqrt{2\log(\alpha_n)s})}{\sqrt{2\log(\alpha_n)s}}\right\rbrace^{n} &=\left\lbrace 1 - \frac{\alpha_{n}^{-s}}{\sqrt{4\pi\log(\alpha_n)s}}\right\rbrace^{n} \\
                                                &=\left\lbrace 1 - \frac{1}{n\sqrt{4\pi\log(n)}}\right\rbrace^{n}.
\end{align*}
Therefore, for $n\geq 2$
$$
\Phi^{n} \left(\sqrt{2\log(\alpha_{n})s}\right)\geq \left\lbrace 1 - \frac{1}{n\sqrt{4\pi\log(n)}}\right\rbrace^{n} \geq \left\lbrace 1 - \frac{1}{2\sqrt{4\pi\log(2)}}\right\rbrace^{2}.
$$
We conclude that
\begin{align*}
&f(n,\rho) \geq  n^{1-1/\rho} \frac{1}{\sqrt{4+2(1/\rho-1)\log(n)}+\sqrt{2(1/\rho-1)\log(n)}} \frac{2}{\sqrt{2\pi}} \left\lbrace 1 - \frac{1}{2\sqrt{4\pi\log(2)}}\right\rbrace^{2}.
\end{align*}
Now \eqref{eq:bounds_rho_high_1} follows immediately from the last display.

We turn now to the proof of \eqref{eq:bounds_rho_high_2}. First, some preliminaries.
Let $U= \Phi\left(Z \sqrt{s} \right)$. Note that $U$ takes values in $[0,1]$. We will calculate the density function of $U$, $f_{U}$. Take $u\in [0,1]$. Then
$$
P\left(U \leq u \right)= P\left(\Phi\left(Z \sqrt{s} \right) \leq u \right) = P\left(Z \leq \Phi^{-1}(u)/\sqrt{s} \right).
$$
Differentiating this last display with respect to $u$ we get
$$
f_{U}(u)= \frac{1}{\sqrt{s}} \frac{\phi\left( \Phi^{-1}(u)/\sqrt{s}\right) }{\phi\left( \Phi^{-1}(u)\right)}.
$$
Simplifying, we get
\begin{align*}
f_{U}(u)= \frac{1}{\sqrt{s}} \exp\left\lbrace \frac{-1}{2s}  [\Phi^{-1}(u)]^{2} + \frac{1}{2}[\Phi^{-1}(u)]^{2} \right\rbrace
&= \frac{1}{\sqrt{s}} \exp\left\lbrace \frac{1}{2}  [\Phi^{-1}(u)]^{2} \left( 1-\frac{1}{s} \right) \right\rbrace
\\
&=
\frac{(\sqrt{2\pi})^{1/s-1}}{\sqrt{s}} \left[\phi\left(  \Phi^{-1}(u)\right)\right]^{1/s-1}.
\end{align*}
Hence, by \eqref{eq:steck}
\begin{equation}
f(n,\rho) = \frac{(\sqrt{2\pi})^{1/s-1}}{\sqrt{s}} \int\limits_{0}^{1} x^{n} \left[\phi\left(  \Phi^{-1}(x)\right)\right]^{1/s-1}dx.
\label{eq:main_pre_bound}
\end{equation}
We claim that
\begin{equation}
\min(x,1-x) \sqrt{\frac{2}{\pi}} \leq \phi\left(  \Phi^{-1}(x)\right)
\label{eq:phi_lin}
\end{equation}
By symmetry it suffices to show that $\phi\left(  \Phi^{-1}(x)\right)\geq \sqrt{\frac{2}{\pi}} x$ for $x\in [0,1/2]$. This follows from the fact that $\phi\left(  \Phi^{-1}(x)\right)$ is concave, and equal to $\sqrt{\frac{2}{\pi}} x$ for $x=0$ and $x=1/2$.

We are now ready to prove \eqref{eq:bounds_rho_high_2}.
The assumption $\rho>1/2$ implies $s>1$ and hence $1/s -1 <0$. Now write
$$
\int\limits_{0}^{1} x^{n} \left[\phi\left(  \Phi^{-1}(x)\right)\right]^{1/s-1}dx = \int\limits_{0}^{1/2} x^{n} \left[\phi\left(  \Phi^{-1}(x)\right)\right]^{1/s-1}dx + \int\limits_{1/2}^{1} x^{n} \left[\phi\left(  \Phi^{-1}(x)\right)\right]^{1/s-1}dx.
$$
Using the lower bound in \eqref{eq:phi_lin} we get
\begin{align*}
&\int\limits_{0}^{1/2} x^{n} \left[\phi\left(  \Phi^{-1}(x)\right)\right]^{1/s-1}dx + \int\limits_{1/2}^{1} x^{n} \left[\phi\left(  \Phi^{-1}(x)\right)\right]^{1/s-1}dx \leq
\\
&
\left(\sqrt{\frac{\pi}{2}}\right)^{1-1/s} \left\lbrace \int\limits_{0}^{1/2} x^{n} x^{1/s-1}dx +  \int\limits_{1/2}^{1} x^{n} (1-x)^{1/s-1}dx\right\rbrace=
\\
&
\left(\sqrt{\frac{\pi}{2}}\right)^{1-1/s} \left\lbrace \left(\frac{1}{2}\right)^{n+1/s} \frac{1}{n+1/s}+  \int\limits_{1/2}^{1} x^{n} (1-x)^{1/s-1}dx \right\rbrace\leq
\\
&
\left(\sqrt{\frac{\pi}{2}}\right)^{1-1/s} \left\lbrace \left(\frac{1}{2}\right)^{n+1/s} \frac{1}{n+1/s}+  B(n+1, 1/s)\right\rbrace.
\end{align*}
Thus
$$
f(n,\rho)\leq \frac{1}{ 2^{1-1/s} \sqrt{s}}  \left\lbrace \left(\frac{1}{2}\right)^{n+1/s} \frac{1}{n+1/s}+  B(n+1, 1/s)\right\rbrace.
$$
Replacing $s=\rho/(1-\rho)$ yields
$$
f(n,\rho)\leq \frac{1} {2^{2-1/\rho}}\sqrt{\frac{1-\rho}{\rho  }}   \left\lbrace \left(\frac{1}{2}\right)^{n+1/\rho -1} \frac{1}{n+1/\rho -1}+  B(n+1, 1/\rho -1)\right\rbrace.
$$
Now, for all $n$,
$
2^{-n-1/s} (n+1/s)^{-1} \leq n^{1-1/\rho}.
$
Recall that
$$
B(n+1, 1/s)= \frac{\Gamma(n+1)\Gamma(1/s)}{\Gamma(n+1+1/s)}.
$$
Using inequality (3.75) in \cite{qi} with $x=n$, $a=1$, $b=1/s+1$, $x_0=1$, we have
$$
\frac{\Gamma(n+1)\Gamma(1/s)}{\Gamma(n+1+1/s)} \leq n^{1-1/\rho} \frac{\Gamma(2)\Gamma(1/s)}{\Gamma(2+1/s)}=n^{1-1/\rho} B(2, 1/s)=n^{1-1/\rho} B(2, 1/\rho -1).
$$
Thus we have
$$
f(n,\rho) \leq  n^{1-1/\rho}\frac{1} {2^{2-1/\rho}}\sqrt{\frac{1-\rho}{\rho  }} \left\lbrace 1+B(2, 1/\rho -1)\right\rbrace,
$$
which is precisely \eqref{eq:bounds_rho_high_2}. This finishes the proof of part (1).

Next we prove part (2).
Since $\rho<1/2$, we have $s<1$ and hence $1/s -1 >0$. Using the lower bound in \eqref{eq:phi_lin} we get
\begin{align*}
&\int\limits_{0}^{1/2} x^{n} \left[\phi\left(  \Phi^{-1}(x)\right)\right]^{1/s-1}dx + \int\limits_{1/2}^{1} x^{n} \left[\phi\left(  \Phi^{-1}(x)\right)\right]^{1/s-1}dx \geq
\\
&
\left(\sqrt{\frac{\pi}{2}}\right)^{1-1/s}   \int\limits_{1/2}^{1} x^{n} (1-x)^{1/s-1}dx =
\\
&
\left(\sqrt{\frac{\pi}{2}}\right)^{1-1/s} \left\lbrace  - \int\limits_{0}^{1/2} x^{n} (1-x)^{1/s-1}dx + B(n+1,1/s)  \right\rbrace \geq
\\
&
\left(\sqrt{\frac{\pi}{2}}\right)^{1-1/s} \left\lbrace - \frac{1}{2^{n}}\int\limits_{0}^{1/2} (1-x)^{1/s-1}dx + B(n+1,1/s)  \right\rbrace =
\\
&
\left(\sqrt{\frac{\pi}{2}}\right)^{1-1/s} \left\lbrace  - \frac{s}{2^{n}}+\frac{s}{2^{n+1/s}} + B(n+1,1/s)  \right\rbrace \geq
\\
&
\left(\sqrt{\frac{\pi}{2}}\right)^{1-1/s} \left\lbrace  -  \frac{s}{2^{n}}+ B(n+1,1/s)  \right\rbrace.
\end{align*}
For all $n$ such that $n\geq (1/\rho -1)\log(n)/\log(2)$,
$
s 2^{-n} \leq 2^{-n}\leq n^{1-1/\rho}.
$
Using inequality (3.75) in \cite{qi} with $x=n$, $a=1$, $b=1/s+1$, $x_0=1$, we have
$$
B(n+1,1/s)=\frac{\Gamma(n+1)\Gamma(1/s)}{\Gamma(n+1+1/s)} \geq n^{-1/s} \Gamma(1/s)= n^{1-1/\rho}\Gamma(1/\rho -1).
$$
Since $\rho<1/2$,
$
\Gamma(1/\rho -1) > \Gamma(1)=1.
$
Thus
$$
\left(\sqrt{\frac{\pi}{2}}\right)^{1-1/s} \left\lbrace  -  \frac{s}{2^{n}}+ B(n+1,1/s)  \right\rbrace \geq n^{1-1/\rho}\left(\sqrt{\frac{\pi}{2}}\right)^{1-1/s} \left\lbrace \Gamma(1/\rho -1) -1 \right\rbrace.
$$
We have shown that
\begin{align*}
f(n,\rho) &\geq  n^{1-1/\rho} \frac{(\sqrt{2\pi})^{1/s-1}}{\sqrt{s}}\left(\sqrt{\frac{\pi}{2}}\right)^{1-1/s} \left\lbrace \Gamma(1/\rho -1) -1 \right\rbrace
\\
&
=
n^{1-1/\rho} \frac{1} {2^{2-1/\rho}}\sqrt{\frac{1-\rho}{\rho  }}\left\lbrace \Gamma(1/\rho -1) -1 \right\rbrace .
\end{align*}
This proves \eqref{eq:bounds_rho_low}.

Finally, turn to the proof of \eqref{eq:bounds_rho_low_2}. Let $h_{n}=n/\left[ {\log(n)}\log(n^{1/s})\right]$ and $\gamma_{n}=\Phi(\sqrt{2\log(h_n)})$. Using \eqref{eq:main_pre_bound} write
\begin{equation}
f(n,\rho) = \frac{(\sqrt{2\pi})^{1/s-1}}{\sqrt{s}} \left\lbrace \int\limits_{0}^{\gamma_{n}} x^{n} \left[\phi\left(  \Phi^{-1}(x)\right)\right]^{1/s-1}dx + \int\limits_{\gamma_{n}}^{1} x^{n} \left[\phi\left(  \Phi^{-1}(x)\right)\right]^{1/s-1}dx \right\rbrace.
\label{eq:last_1}
\end{equation}
Since $\rho<1/2$ we have $s<1$ and hence $1/s-1 >0$. Then $\left[\phi\left(  \Phi^{-1}(x)\right)\right]^{1/s-1}$ is a decreasing function over $(1/2,1)$. Since for sufficiently large $n$ it holds that $\gamma_{n}> 1/2$, we have that
\begin{align}
&\int\limits_{\gamma_{n}}^{1} x^{n} \left[\phi\left(  \Phi^{-1}(x)\right)\right]^{1/s-1}dx \leq
\nonumber
\\
&  \left[\phi\left(  \Phi^{-1}(\gamma_{n})\right)\right]^{1/s-1} \int\limits_{\gamma_{n}}^{1} x^{n} dx=
\nonumber
\\
& \frac{1}{(\sqrt{2\pi})^{1/s-1}} \left[ \exp(-\log(n) + \log({\log(n)}\log(n^{1/s}))\right]^{1/s-1} \frac{(1 - \gamma_{n}^{n+1})}{n+1}
\nonumber=
\\
& \frac{1}{(\sqrt{2\pi})^{1/s-1}} n^{1-1/s} \left[{\log(n)}\log(n^{1/s})\right]^{1/s -1} \frac{(1 - \gamma_{n}^{n+1})}{n+1}\leq
\nonumber
\\
& \frac{1}{(\sqrt{2\pi})^{1/s-1}} n^{-1/s} \left[{\log(n)}\log(n^{1/s})\right]^{1/s -1}=
\nonumber
\\
& \frac{1}{(\sqrt{2\pi})^{1/\rho-2}} n^{1-1/\rho} \left[(1/\rho - 1){\log(n)^{2}}\right]^{1/\rho -2}.
\label{eq:last_2}
\end{align}

Note that
\begin{align}
\int\limits_{0}^{\gamma_{n}} x^{n} \left[\phi\left(  \Phi^{-1}(x)\right)\right]^{1/s-1}dx \leq  \gamma_{n}^{n} \int\limits_{0}^{1}  \left[\phi\left(  \Phi^{-1}(x)\right)\right]^{1/s-1}dx
\label{eq:last_3}
\end{align}

By Birnbaum's inequality
\begin{align*}
\gamma_{n}^{n}&\leq  \left( 1 - \frac{\phi(\sqrt{2\log(h_{n}})}{\sqrt{4 + {2\log(h_{n})}}+ \sqrt{2\log(h_{n})}}\right)^{n}
\\
&=
\left( 1 - \frac{{\log(n)} \log(n^{1/s})}{n \sqrt{2\pi}(\sqrt{4 + {2\log(h_{n})}}+ \sqrt{2\log(h_{n})})}\right)^{n}.
\end{align*}
Now, since $h_{n}= n/\left[ {\log(n)}\log(n^{1/s})\right]$, for all sufficiently large $n$
$$
\frac{{\log(n)} \log(n^{1/s})}{n \sqrt{2\pi}(\sqrt{4 + {2\log(h_{n})}}+ \sqrt{2\log(h_{n})})} \geq  \frac{\log(n^{1/s})}{n}.
$$
Thus
\begin{align*}
\gamma_{n}^{n} \leq \left( 1- \frac{ \log(n^{1/s})}{n}\right)^{n}.
\end{align*}
Also, for all sufficiently large $n$
$$
\left( 1- \frac{ \log(n^{1/s})}{n}\right)^{n/\log(n^{1/s})} \leq \exp(-1)
$$
and hence
$$
\gamma_{n}^{n} \leq \exp(- \log(n^{1/s})) = n^{-1/s} = n^{1-1/\rho}.
$$
This together with \eqref{eq:last_1}, \eqref{eq:last_2} and \eqref{eq:last_3} implies that, for all sufficiently large $n$,
$$
f(n,\rho) \leq n^{1-1/\rho}  \sqrt{\frac{1-\rho}{\rho}} \left[(1/\rho - 1){\log(n)^{2}}\right]^{1/\rho -2},
$$
This proves \eqref{eq:bounds_rho_low_2}.


\end{proof}
The case $\rho=1/2$ is excluded from Theorem \ref{cotainf}, since it is well known, see \cite{Owen} for example, that $f(n,1/2)=1/(n+1)$. Theorem \ref{cotainf} implies that, for $\rho >0$ and large $n$, $f(n,\rho)$ behaves essentially like $n^{1-1/\rho}$, except for logarithmic factors.

Recall that the probability that a $k$-homogeneous polynomial, $k\geq 2$, attains a relative maximum at a given vertex of the $n$-dimensional simplex is given by the probability that the standard normal random variables $X_{1}(a),\dots,X_{n}(a)$ are all positive.
By \eqref{eq:rho_n}, the covariance matrix of the multivariate normal vector $\mathbf{X(a)}=(X_1(a),\ldots ,X_n(a))$ is given
\[
A=\left(
  \begin{array}{ccccc}
    1 & \rho_n & \rho_n & \ldots & \rho_n \\
    \rho_n & 1 & \rho_n & \ldots & \rho_n \\
    \vdots & \vdots & \ddots & \ldots & \rho_n \\
    \rho_n & \rho_n & \rho_n & \ddots &  \rho_n\\
     \rho_n & \rho_n & \rho_n & \ldots & 1 \\
  \end{array}
\right),
\]
where $\rho_{n}=(nk+(k-1))/(n(k+1)+(k-1))$ satisfies $\rho_{n} >1/2$ and $\rho_{n} \to k/(k+1)$.
Thus, using the notation of Theorem \ref{cotainf}, for $k\geq 2$, the probability that a $k$-homogeneous polynomial attain a relative maximum at a given vertex of the $n$-dimensional simplex is given by $f(n,\rho_{n})$. By Theorem \ref{cotainf} for large enough $n$, we have
$$
\frac{c}{\sqrt{\log(n)}}\leq \frac{f(n,\rho_{n})}{n^{1/\rho_{n}-1}}\leq c
$$
where $c>0$ is a universal constant.

\section{Attaining maxima as $n$ grows}\label{AM}

In this section we will consider the probability of a $k$-homogeneous ($k\geq 2$) polynomial attaining relative maxima at {\it any} vertex of the simplex. Denote by $A_i$ the set of polynomials attaining a relative
maximum at the $i^{\text{th}}$ vertex, and by $W(A_i)$ the measure of this set. We know ---because of rotation-invariance--- that $W(A_i)=W(A_j)=f(n,\rho_{n})$ for all $i, j$.
Note that if the $A_i$'s were independent events (they are not), the probability of attaining a relative maximum at {\it some} vertex of the simplex would be
$$
W\left(\bigcup_{i=1}^{n+1}A_i\right)=1 - W\left(\bigcap_{i=1}^{n+1}A_i^c \right)=1 - W(A_i^c)^{n+1}=1 - (1 - f(n,\rho_{n}))^{n+1}
$$
We know that for large enough $n$, $f(n,\rho_{n}) \geq \frac{\beta_n}{n+1}$, where
$$
\beta_n=c \frac{n^{\frac{-n}{nk+(k-1)}}}{\sqrt{\log n}}(n+1).
$$
But we have $\beta_n \longrightarrow \infty$, and $\frac{\beta_n}{n+1} \longrightarrow 0$, so
$$
1-(1-f(n,\rho_{n}))^{n+1} \geq 1-\left(1-\frac{\beta_n}{n+1}\right)^{n+1} \longrightarrow 1.
$$

Since the $A_i$'s are non-independent events, in order to bound the difference between $1 - (1 - f(n,\rho_{n}))^{n+1}$ and $W\left(\bigcup_{i=1}^{n+1}A_i\right)$ we will
consider the $(n+1)n \times (n+1)n$ matrices
\[
\Sigma_0=\left(
  \begin{array}{ccccc}
    A & 0 & 0 & \ldots & 0 \\
   0 & A & 0 & \ldots & 0 \\
    \vdots & \vdots & \ddots & \ldots & 0 \\
   0 & 0 & 0 & \ddots &  0\\
     0 &0 & 0 & \ldots & A \\
  \end{array}
\right)
\]
and
\[
\Sigma_\ve=\left(
  \begin{array}{ccccc}
    A & M_{\ve} & M_{\ve} & \ldots & M_{\ve} \\
    M_{\ve} & A & M_{\ve} & \ldots & M_{\ve} \\
    \vdots & \vdots & \ddots & \ldots & M_{\ve} \\
    M_{\ve} & M_{\ve} & M_{\ve}  & \ddots &  M_{\ve}\\
     M_{\ve} & M_{\ve} & M_{\ve} & \ldots & A \\
  \end{array}
\right)
\]
where $A$ is the $n\times n$ covariance matrix of $\mathbf{X(a)}$, and $M_{\ve}$ is the matrix whose $(i,j)$-th component is $\langle X_i(a), X_j(b) \rangle$, where $a$ and $b$ are
different vertices of the simplex. Recall that each component of $M_{\ve}$ is bounded in absolute value by $\ve=\frac{1}{n^{k-2}} \frac{1}{2k-1} \left(\frac{1}{n} + (k-1)\right)$.
We need the following lemma.
\begin{lemma}\label{Ainverse}
The inverse of $A$ is
\[
B=\left(
  \begin{array}{ccccc}
    \alpha & \beta & \beta & \ldots & \beta \\
    \beta & \alpha & \beta & \ldots & \beta \\
    \vdots & \vdots & \ddots & \ldots & \beta \\
    \beta & \beta & \beta & \ddots &  \beta\\
     \beta & \beta & \beta & \ldots & \alpha \\
  \end{array}
\right)
\]
where $\alpha$ and $\beta$ are such that, as $n \longrightarrow \infty$, $\alpha \longrightarrow k+1$ and $(n-1)|\beta| \longrightarrow k+1$.
\end{lemma}

\begin{proof}
For $AB=I$ we need
\begin{align*}
\alpha + (n-1)\rho_n\beta =& 1 \\
\text{and }\quad \rho_n\alpha + \left[(n-2)\rho_n+1\right]\beta =& 0.
\end{align*}
Thus
\begin{align*}
\alpha =& \frac{(n-2)\rho_n + 1}{(n-2)\rho_n + 1 - \rho_n^2 (n-1)} \\
\text{and }\quad \beta =& \frac{-\rho_n}{(n-2)\rho_n + 1 - \rho_n^2 (n-1)},
\end{align*}
which, for our value of $\rho_n$ is
\begin{align*}
\alpha =& \frac{n^3(k^2+k)-n^2(3k+1)-n(k^2-2k+3)+2k-2}{n^3k-n^2(k^2+k+1)+2n(k+1)+k^2-1} \\
\text{and }\quad \beta =& \frac{n^2(k^2+k)+n(2k^2-k-1)+(k-1)^2}{-n^3k-n^2k}.
\end{align*}
\end{proof}

We now have the following.
\begin{theorem}\label{T}
The probability that a $k$-homogeneous polynomial $P:\zR^n \longrightarrow \zR$ of degree $k > 4$ attains a relative maximum at a vertex of the $n$-dimensional simplex tends to one
as the dimension $n$ tends to infinity.
\end{theorem}

\begin{proof}
Consider the $n(n+1)$ random variables
$$
X_1(a_1),\ldots ,X_n(a_1),X_1(a_2), \ldots ,X_n(a_2),\ldots , X_1(a_{n+1}),\ldots , X_n(a_{n+1})
$$
where $a_{1}, a_{2}, \ldots ,a_{n+1}$ are the $n+1$ vertices of the simplex $\Delta$. These are $N(0,1)$ random variables with covariance matrix $\Sigma_\ve$. If the events $A_1, \ldots , A_{n+1}$ were
independent, the covariance matrix would be $\Sigma_0$. We denote by $G_\ve$ the gaussian measure on $\zR^{n(n+1)}$ with covariance matrix $\Sigma_\ve$ and by $G_0$ the gaussian measure
on $\zR^{n(n+1)}$ whose covariance matrix is $\Sigma_0$. We are interested in the probability of the event $\bigcup_{i=1}^{n+1}A_i$: ``for some vertex $a_{1}, a_{2}, \ldots ,a_{n+1}$,
all $X_1,X_2,\ldots ,X_n$ (corresponding to that vertex) are positive''. We calculate the probability of the complement $\bigcap_{i=1}^{n+1}A_i^c$: ``for all vertices $a_{1}, a_{2}, \ldots ,a_{n+1}$,
some $X_j$ (corresponding to each vertex) is negative''. We do this by integrating the gaussian measure $G_\ve$ of covariance $\Sigma_\ve$ over the subset $\Omega$ of
$$
\zR^{n(n+1)} = \left\{(\mathbf{z}_1,\ldots ,\mathbf{z}_{n+1}) : \mathbf{z}_i=(z_{i1},\ldots ,z_{in)}) \in \zR^{n} \right\}
$$
given by
$$
\Omega = \left\{(\mathbf{z}_1,\ldots ,\mathbf{z}_{n+1}) \in \zR^{n(n+1)} : \text{ for all }i\leq n+1 \text{ there is a }j\leq n : z_{ij}\leq 0 \right\},
$$
and analogously for $G_0$. Thus
\begin{align*}
\left|1 - (1 - f(n,\rho_{n}))^{n+1} - W\left(\bigcup_{i=1}^{n+1}A_i\right)\right| &= \left|1 - (1 - f(n,\rho_{n}))^{n+1} - 1 + W\left(\bigcap_{i=1}^{n+1}A_i^c \right)\right| \\
                                                                        &= \left|W\left(\bigcap_{i=1}^{n+1}A_i^c \right) - (1 - f(n,\rho_{n}))^{n+1}\right| \\
                                                                        &= \left| \int_\Omega dG_\ve - dG_0 \right|.
\end{align*}
This is bounded by the total variation distance between the measures $G_\ve$ and $G_0$.
By the Devroye-Mehrabian-Reddad theorem \cite{DMR} this distance will in turn be bounded by
$$
\frac{3}{2}\left\Vert \Sigma_\ve \Sigma_0^{-1} - I \right\Vert_F,
$$
where $\Vert \cdot \Vert_F$ is the Frobenius norm.
Thus consider
\[
\Sigma_\ve \Sigma_0^{-1}-I=\left(
  \begin{array}{ccccc}
    0 & M_{\ve}B & M_{\ve}B & \ldots & M_{\ve}B \\
    M_{\ve}B & 0 & M_{\ve}B & \ldots & M_{\ve}B \\
    \vdots & \vdots & \ddots & \ldots & M_{\ve}B \\
    M_{\ve}B & M_{\ve}B & M_{\ve}B  & \ddots &  M_{\ve}B\\
     M_{\ve}B & M_{\ve}B & M_{\ve}B & \ldots & 0 \\
  \end{array}
\right).
\]
Each of the $(n+1)^2-(n+1)$ non-zero blocks is an $n\times n$ matrix whose entries are inner products between a row of $M_\ve$ and a column of $B$. Thus each of these entries is bounded
by $\ve \Vert (\alpha , \beta , \ldots , \beta )\Vert_1$. We then have
$$
\Vert \Sigma_\ve^{-1} \Sigma_0 - I \Vert_F \leq ((n+1)^2-(n+1))n^2 \ve^2 \Vert (\alpha , \beta , \ldots , \beta )\Vert_1^2 < (n+1)^2 n^2 \ve^2 (\alpha + (n-1)|\beta|)^2,
$$
which for large $n$ is bounded by
$$
C n^4 \ve^2 (k+1)^2 \leq C^{\prime} \frac{(k+1)^2}{n^{2k-8}},
$$
where $C$ and $C^{\prime}$ are positive constants.
Thus if $k\geq 5$, for large $n$,
\begin{align*}
\left|1 - W\left(\bigcup_{i=1}^{n+1}A_i\right)\right| &\leq \left|(1 - f(n,\rho_{n}))^{n+1}\right| + \left|1 - (1 - f(n,\rho_{n}))^{n+1} - W\left(\bigcup_{i=1}^{n+1}A_i\right)\right| \\
                                                      &\leq \left|(1 - f(n,\rho_{n}))^{n+1}\right| + C^{\prime} \frac{(k+1)^2}{n^{2k-8}} \to 0.
\end{align*}
This finishes the proof of the theorem.
\end{proof}

\bibliographystyle{plain}

\end{document}